\documentclass[12pt]{amsart}

\usepackage{amsfonts,amssymb,stmaryrd,amscd,amsmath,latexsym,amsbsy}

\newtheorem{theorem}{Theorem}[section]
\newtheorem{lemma}[theorem]{Lemma}
\newtheorem{proposition}[theorem]{Proposition}
\newtheorem{corollary}[theorem]{Corollary}
\theoremstyle{definition}

\newtheorem{remark}[theorem]{Remark}

\newcommand{\Tr}{\text{Tr}}%{\text{Tr}\,}

\newcommand{\Id}{\text{Id}}
\newcommand{\Aut}{\text{Aut}}

\newcommand{\Rep}{\text{Rep}}

\renewcommand{\O}{\mathcal{O}}

\newcommand{\C}{\mathcal{C}}

\newcommand{\ben}{\begin{enumerate}}
\newcommand{\een}{\end{enumerate}}

\theoremstyle{plain}

\newtheorem*{sol}{Solution}

\theoremstyle{definition}

\theoremstyle{remark}

\newcommand{\solu}[1]{\begin{sol}{\bf (\ref{#1})}}

\def\C{\mathcal{C}}
\def\D{\mathcal{D}}

\def\Aut{\mathop{\mathrm{Aut}}\nolimits}

\def\B{\mathcal{B}}

\def\O{\mathcal{O}}

\def\Vec{\mathrm{Vec}}

\def\k{\mathbf{k}}

\def\Rep{\mathop{\mathrm{Rep}}\nolimits}

\def\FPdim{\mathop{\mathrm{FPdim}}\nolimits}

\begin{document}

\title{On faithfulness of the lifting for Hopf algebras and fusion categories}

\author{Pavel Etingof}
\address{Department of Mathematics, Massachusetts Institute of Technology,
Cambridge, MA 02139, USA}
\email{etingof@math.mit.edu}

\begin{abstract} We use a version of Haboush's theorem over complete local Noetherian rings to prove faithfulness of the lifting for semisimple cosemisimple Hopf algebras and separable (braided, symmetric) fusion categories from characteristic $p$ to characteristic zero (\cite{ENO}, Section 9), showing that, moreover, any isomorphism between such structures can be reduced modulo $p$. This  fills a gap in \cite{ENO}, Subsection 9.3. We also show that lifting of semisimple cosemisimple Hopf algebras is a fully faithful functor, and prove that lifting induces an isomorphism on Picard and Brauer-Picard groups. Finally, we show that a subcategory or quotient category of a separable multifusion category is separable (resolving an open question from \cite{ENO}, Subsection 9.4), and use this to show that certain classes of tensor functors between lifts of separable categories to characteristic zero can be reduced modulo $p$. 
\end{abstract}

\maketitle 

\centerline{\bf To Alexander Kirillov Jr. on his 50th birthday with admiration} 

\section{Introduction}

Let $k$ be an algebraically closed field of characteristic $p>0$,  
$W(k)$ be its ring of Witt vectors, $I=(p)\subset W(k)$ the maximal ideal, 
$K$ the fraction field of $W(k)$, $\overline{K}$ its algebraic closure. 
In \cite{EG} it is shown that any semisimple cosemisimple Hopf algebra 
over $k$ has a unique (up to an isomorphism) lift over $W(k)$. 
In \cite{ENO}, Section 9, this result is extended to separable (braided, symmetric) fusion categories, i.e., those of nonzero global dimension.
\footnote{In \cite{ENO}, separable fusion categories are called nondegenerate. But this terminology is confusing since 
for braided fusion categories, this term is also used in an entirely different sense: to refer to categories with trivial M\"uger center. 
For this reason, we adopt a better term ``separable" introduced in \cite{DSS}.} 

Moreover, in \cite{ENO}, Subsection 9.3 
it is claimed that lifting is faithful, i.e., if liftings of two Hopf 
algebras are isomorphic over $\overline{K}$ then these Hopf algebras 
are isomorphic, and similarly for categories (Theorem 9.6, 
Corollary 9.10). This is used in a number of subsequent papers. 

However, it recently came to my attention that the proofs of these faithfulness results given in \cite{ENO}, Subsection 9.3, are incomplete. 
Namely, the proof of Lemma 9.7 (used in the proof of Theorem 9.6) says that by Nakayama's lemma, it suffices to check the finiteness 
of a certain morphism $\phi$ of schemes over $W(k)$ modulo the maximal ideal $I$ (i.e., over $k$). But it is, in fact, not clear how this 
follows from Nakayama's lemma. Namely, finiteness over $k$ does imply 
finiteness over $W(k)/I^N$ for any $N\ge 1$, but this is not sufficient to
conclude finiteness over $W(k)$. \footnote{E.g., $K$ is not module-finite 
over $W(k)$, even though it becomes module-finite (in fact, zero) upon reduction modulo $I^N$, and is also finite over $K$.}  
In fact, the reductivity of the group of twists must be used in the proof, see Remark \ref{finiteness}. 

The main goal of this paper is therefore to provide complete proofs of the results on faithfulness of the lifting. 
This may be done by using results on geometric reductivity and power reductivity of reductive groups over rings, see \cite{Se},\cite{FK}.
We also prove results on integrality of stabilizers of liftings, and show that lifting is a fully faithful functor for semisimple cosemisimple Hopf algebras, and defines an isomorphism of Brauer-Picard and Picard groups of (braided) separable fusion categories.  Finally, we 
prove that subcategories and quotient categories of separable categories are separable (resolving a question from \cite{ENO}, 
Subsection 9.4), and use this to prove that certain types of tensor functors between liftings of separable categories descend to positive characteristic. 

The paper is organized as follows. In Section 2 we describe algebro-geometric preliminaries, i.e., the results on geometric reductivity and power reductivity and their applications. 
In Section 3 we apply these results to proving faithfulness of the lifting and integrality of stabilizers 
for semisimple cosemisimple Hopf algebras and  prove that lifting of such Hopf algebras is a fully faithful functor. In Section 4 we generalize the results of Section 3 to tensor categories and tensor functors, thus providing complete proofs of the results of \cite{ENO}, Subsection 9.3 and 
\cite{EG2}, Theorem 6.1. Also, we apply these results to show that the Brauer-Picard and Picard groups of (braided) separable fusion categories are preserved by lifting. Finally, in Section 5 we show that 
a subcategory and quotient category of a separable category is separable, and apply it to prove a descent result for 
tensor functors between liftings of separable categories. 
  
{\bf Acknowledgements.} I am very grateful to Brian Conrad for useful discussions and help with writing Section 2, and to Davesh Maulik for comments on this section. I thank W. van der Kallen for Remark \ref{vdkrem}, D. Nikshych for Remark \ref{nik}, and V. Ostrik for comments on a draft of this paper. This work was supported by the NSF
grant DMS-1502244.

\section{Auxiliary results from algebraic geometry}

In this section we collect some auxiliary results from algebraic geometry that we will use below. 

\subsection{Power reductivity} 

Let $k$ be an algebraically closed field of characteristic $p>0$,  
$W(k)$ be its ring of Witt vectors, $I=(p)\subset W(k)$ the maximal ideal, 
$K$ the fraction field of $W(k)$, $\overline{K}$ its algebraic closure. 

If $\bold X$ is a scheme over a ring $R$ and $R'$ is a commutative $R$-algebra, then $\bold X_{R'}$ will denote the base change of $\bold X$ from $R$ to $R'$. 

By a {\em reductive group} over a commutative ring $\k$ we will mean a smooth affine group scheme with connected fibers, as in \cite{SGA3}. Such a group $\bold G$ is split when it contains a split fiberwise maximal $\k$-torus as a closed $\k$-subgroup. In our applications, $\bold G$ will always be split, and, in fact, will be a quotient of a product of general linear groups by a central torus. 

We start with restating a result from \cite{FK}, which is a combination of Proposition 6 and Theorem 12 in \cite{FK}. 

We call a ring homomorphism $\psi: B\to C$ {\it power-surjective} if for any element $c\in C$, some positive integer power $c^N$ 
belongs to ${\rm Im}\psi$.  

\begin{proposition}\label{FKprop} 
Let $\k$ be a commutative Noetherian ring and let $\bold G$ be a split
reductive group over $\k$. Let $A$ be a finitely generated commutative
$\k$-algebra on which $\bold G$ acts rationally through algebra automorphisms. If $J$ is
a $\bold G$-invariant ideal in $A$, then the map induced by reducing mod $J$:
$A^{\bold G}\to (A/J)^{\bold G}$ 
is power-surjective.
\end{proposition} 

In \cite{FK} this property of $\bold G$ is called {\it power reductivity}. 
Note that it is not assumed in \cite{FK} that $\k$ is Noetherian, but we will use Proposition \ref{FKprop} only for Noetherian (in fact, complete local) rings $\k$. 

\begin{remark}\label{vdkrem} It was explained to us by W. van der Kallen that power reductivity of $\bold G$ (i.e., Proposition \ref{FKprop}) over complete local Noetherian rings $\k$ follows from \cite{T}, Theorem 3.8. Namely, 
from this theorem one easily gets property (INT) of \cite{FK} and then power reductivity.
This is discussed in Section 2.4 and Theorem 2.2 of \cite{vdK}. 
More precisely, this paper assumes a base field, but this assumption is not essential.
Compare also with Theorem 16.9 in \cite{G} and Lemma 2.4.7 in \cite{Sp}.
\end{remark}

\subsection{Faithfulness of lifting for reductive group actions} 

We will need the following proposition, which is sufficient to justify the main results of \cite{ENO}, Subsection 9.3. 

\begin{proposition}\label{ag}
Let $\bold V$ be a rational representation of a split reductive group $\bold G$ on an affine space defined over $W(k)$. 
Let $v_1,v_2\in \bold V(W(k))$. Assume that the $\bold G$-orbits of the reductions 
$v_1^0,v_2^0\in \bold V(k)$ are closed and disjoint. Then $v_1,v_2$ are not conjugate under 
$\bold G(\overline{K})$. 
\end{proposition}  

\begin{proof} 
This follows from \cite{Se}, Theorem 3, part (ii), for $R=W(k)$, 
$X=\bold V$ and $Y=\bold V/\bold G:={\rm Spec}{\mathcal O}(\bold V)^{\bold G}$. Namely, 
this theorem says that $v_1^0,v_2^0$ are {\it distinct} points of $Y(k)$. 
This implies that there exists a ${\bold G}$-invariant polynomial 
$f\in {\mathcal O}(\bold V)$ such that $f(v_1^0)\ne f(v_2^0)$ in $k$. But $f(v_i^0)$ 
are the reductions of $f(v_i)$ mod $I$, so $f(v_1)\ne f(v_2)$ in $W(k)$. 
But then $v_1,v_2$ cannot be conjugate under $\bold G(\overline{K})$. 

Here is another proof, using Proposition \ref{FKprop}, for $\k=W(k)$, $A={\mathcal O}(\bold V)$ and $J=I{\mathcal O}(\bold V)$. 
Proposition \ref{FKprop} says that for any $f\in {\mathcal O}(\bold V_k)^{\bold G(k)}$, some power $f^N$ of $f$ lifts to a
$\bold G$-invariant $h\in {\mathcal O}(V)$. Now, by Haboush's theorem (\cite{Ha}), we can choose $f$ so that $f(v_1^0)=0$ and $f(v_2^0)=1$. 
Then $h(v_1)\in I$ and $h(v_2)\in 1+I$, hence $h(v_1)\ne h(v_2)$, and $v_1,v_2$ cannot be conjugate under $\bold G(\overline{K})$.  
\end{proof} 

\begin{remark}
The closedness assumption for $\bold G$-orbits in Proposition \ref{ag} cannot be removed.
For instance, let $\bold G=\Bbb G_m$ and $\bold V=\Bbb A^1$ be the tautological representation of $\bold G$. 
Take $v_1=p$ and $v_2=1$. Then $v_1^0=0$ and $v_2^0=1$, so their $\bold G$-orbits are disjoint. 
On the other hand, $v_1$ and $v_2$ are conjugate by the element $p\in \bold G$. Note that the orbit 
of $v_2^0$ is $\Bbb A^1\setminus \lbrace{0\rbrace}$, hence not closed. 

The reductivity assumption cannot be removed, either. Namely, let  $\bold G=\Bbb G_a$ be the group of translations
$x\mapsto x+b$, and $\bold V$ be the 2-dimensional representation of $\bold G$ on  linear (not necessarily homogeneous) functions of $x$. 
Let $v_1=px$ and $v_2=1+px$.  Then $v_1^0=0$, $v_2^0=1$. These are $\bold G$-invariant vectors, so their orbits are closed and disjoint. 
Still, $v_1$ is conjugate to $v_2$ by the transformation $x\to x+\frac{1}{p}$. 
\end{remark}

\subsection{Integrality of the stabilizer}

Let $\bold V$ be a rational representation of a split reductive group $\bold G$ on an affine space defined over $W(k)$.
Let $v\in \bold V(W(k))$ and $v^0\in \bold V(k)$ be its reduction modulo $p$. 
Let $\bold S\subset \bold G$ be the scheme-theoretic stabilizer of $v$, and 
$S=\bold S(\overline K)$ be the stabilizer of $v$ in $\bold G(\overline K)$. 
Let $\bold S_0=\bold S_k\subset \bold G_k$ be the scheme-theoretic stabilizer of $v^0$, and $S_0=\bold S_0(k)=\bold S(k)$
be the stabilizer of $v^0$ in $\bold G(k)$. 

\begin{proposition}\label{ag2} 
Assume that:

(i)  the $\bold G$-orbit of $v^0$ is closed; 

(ii) $\bold S_0$ is finite and reduced; 

(iii) there is a lifting map $i: S_0\to S$ such that $i(S_0)\in \bold G(W(k))$ and  
the reduction of $i(g_0)$ modulo $p$ equals $g_0$ for any $g_0\in S_0$. 

Then $i$ is an isomorphism. 
\end{proposition} 

\begin{proof} 
 Let $\bold U$ be a defining representation of $\bold G$, presenting it as a closed subgroup 
 of $GL(\bold U)$. The main part of the proof is showing that the matrix elements 
of any $g\in S$ in some basis of $\bold U(W(k))$ are integral over $W(k)$. 
To this end, we want to construct a lot of $S$-invariants in ${\mathcal O}(\bold U)$, so that 
 ${\mathcal O}(\bold U)$ is integral over the subalgebra generated by these invariants. 
 
 Let $\bold X=\bold Gv^0\subset \bold V_k$ 
be the $\bold G$-orbit of $v^0$ over $k$, which is closed by (i). Since $\bold S_0$ is reduced by (ii), 
the natural morphism $\bold G_k/\bold S_0\to \bold V_k$ given by $g\mapsto gv^0$ defines an isomorphism 
$\bold G_k/\bold S_0\cong \bold X$. Therefore, we have 
${\mathcal O}(\bold X\times \bold U_k)^{\bold G(k)}={\mathcal O}(\bold U_k)^{S_0}$. 
Thus, given a homogeneous $f\in {\mathcal O}(\bold U_k)^{S_0}$ of some degree $\ell$, we may view $f$ as a 
$\bold G$-invariant regular function on $\bold X\times \bold U_k$. 
By Proposition \ref{FKprop}, some power $f^N$ lifts to a $\bold G$-invariant polynomial
$h_f$ on $\bold V\times \bold U$, homogeneous of degree $N\ell$ in the second variable (as
${\mathcal O}(\bold X\times \bold U_k)$ is a quotient of ${\mathcal O}(\bold V\times \bold U)$ by a
$\bold G$-invariant ideal). Then $h_f(v,\cdot)$ is a lift over $W(k)$ of $f^N(v^0,\cdot)$, which is an
$S$-invariant element of $\mathcal O(\bold U)$, homogeneous of degree
$N\ell$.

Since by (ii) $S_0$ is finite, Noether's theorem allows us to pick a finite collection of homogeneous generators $f_1,...,f_m\in 
{\mathcal O}(\bold U_k)^{S_0}$, of some degrees $\ell_1,...,\ell_m$. Let $h_{f_j}$ be a lift of $f_j^{N_j}$ as above, $j=1,...,m$.

\begin{lemma}\label{mfi} The algebra $\mathcal O(\bold U)$ is module-finite over $W(k)[h_{f_1},...,h_{f_m}]$. 
\end{lemma} 

\begin{proof} First, ${\mathcal O}(\bold U_k)^{S_0}=k[f_1,...,f_m]$ is module-finite over $k[f_1^{N_1},...,f_m^{N_m}]$.
Also, by Noether's theorem, ${\mathcal O}(\bold U_k)$ is module-finite over 
${\mathcal O}(\bold U_k)^{S_0}$. Thus, ${\mathcal O}(\bold U_k)$ is module-finite over 
$k[f_1^{N_1},...,f_m^{N_m}]$. 

Let $w_1^0,...,w_r^0$ be homogeneous module generators of ${\mathcal O}(\bold U_k)$ over
$k[f_1^{N_1},...,f_m^{N_m}]$, of degrees $p_1,...,p_r$. Let $w_j$ be any homogeneous lifts of $w_j^0$ over $W(k)$, $j=1,...,r$. We claim that $w_1,...,w_r$ are module generators of ${\mathcal O}(\bold U)$ over 
$R:=W(k)[z_1,...,z_m]$, where $z_j$ acts by multiplication by $h_{f_j}$. Indeed, set $\deg(z_j)=N_j\ell_j$. 
For each degree $s$, we have a natural map $\psi_s: R[s-p_1]\oplus...\oplus R[s-p_r]\to {\mathcal O}(\bold U)[s]$, where 
$[s]$ denotes the degree $s$ part; namely, $\psi_s(b_1,...,b_r)=\sum_{j=1}^r b_jw_j$. 
The map $\psi_s$ is a morphism of free finite rank $W(k)$-modules, and the reduction of 
$\psi_s$ mod $I$ is surjective, since $w_1^0,...,w_r^0$ generate ${\mathcal O}(\bold U_k)$ as a module over 
$k[f_1^{N_1},...,f_m^{N_m}]$. Hence, $\psi_s$ is surjective as well, which implies the lemma. 
\end{proof}

By Lemma \ref{mfi}, for any linear function $F\in \bold U^*(W(k))$ (where $\bold U^*$ is the dual representation to $\bold U$), we have
\begin{equation}\label{integ}
F^n+a_1F^{n-1}+...+a_n=0,
\end{equation}
where $a_j\in W(k)[h_{f_1},...,h_{f_m}]$. 

Let $g\in S$. Acting by $g$ on equation \eqref{integ}, and using that $ga_j=a_j$ (since by construction $h_{f_j}$ are $S$-invariant), we get that $gF$ also satisfies \eqref{integ}. This means that if $u\in \bold U(W(k))$ then $gF(u)\in \overline{K}$ is integral over $W(k)$ (as $a_j(u)\in W(k)$). Since the integral closure of $W(k)$ in $\overline{K}$ is $\overline{W(k)}$, the ring 
of integers  in $\overline{K}$, we get that $gF(u)\in \overline{W(k)}$, hence $gF\in \bold U^*(\overline{W(k)})$. Thus $gF\in \bold U^*(Q)$, where $Q\subset \overline{W(k)}$ is a finite extension of $W(k)$. So, the matrix elements of $g$ in some free $W(k)$-basis of $\bold U$ belong to $Q$. Thus, $g$ is a lift of some $g_0\in S_0$ over $Q$. 
Using (iii), consider the element $\bar g:=gi(g_0)^{-1}\in S\cap \bold G(Q)=\bold S(Q)$, and let $\bar g_N\in \bold S(Q/(p^N))$ be the reduction of $\bar g$ modulo $p^N$. By construction, the image of $\bar g$ in $\bold G(k)$ is $1$. 
Since $\bold S_0$ is reduced by (ii), this implies that $\bar g_N=1$ for all $N$.
Thus, $\bar g=1$ and $g=i(g_0)$, as desired.   
\end{proof} 

\begin{remark} Condition (i) in Proposition \ref{ag2}  is essential.  
Indeed, take $\bold G=GL(2)$ acting on 
$\bold V=V_2\oplus V_1$, where $V_m$ is the natural representation of $\bold G$ on homogeneous polynomials of two variables $x,y$ of degree $m$. 
Take $v=x_1(y_1-px_1)+y_2\in \bold V(W(k))$. Then $v^0=x_1y_1+y_2$. If $gv^0=v^0$ then $g$ preserves $y_2$, hence $y_1$, so it preserves $x_1$. Thus $g=1$, hence $\bold S_0=1$. 
On the other hand, $S=\lbrace{1,s\rbrace}$, where $s(y)=y$, $s(x)=\frac{1}{p}y-x$.  

The reductivity of $\bold G$ is essential as well. To see this, take $\bold G={\rm Aff}(1)$, the group of affine linear transformations $(a,b)$ given by $x\mapsto ax+b$, $a\ne 0$, 
and take $\bold V=\bold Q\oplus \bold U$, where $\bold Q$ is the space of quadratic (not necessarily homogeneous) 
functions of $x$, and $\bold U$ is the 1-dimensional representation with basis vector $z$ defined by $(a,b)(z)=a^2z$. 
We can then take $v=x(1-px)+z$, so that $v^0=x+z$. Then, as before, $\bold S_0=1$, but $S=\lbrace{1,s\rbrace}$, where $s(x)=-x+\frac{1}{p}$. 
 Note that the $\bold G$-orbit of $v^0$ is closed in this case.   
\end{remark}

\subsection{Finiteness of the orbit map} 

The results of this subsection are not needed for the proof of the main results. They are only included to justify Lemma 9.7 of \cite{ENO}, whose original proof is incomplete. 

We keep the setting of Proposition \ref{ag2}.

\begin{lemma}\label{finetale} (i) For all $r\ge 1$, $\mathcal O(\bold S)/p^r\mathcal O(\bold S)$ is a free $W(k)/(p^r)$-module of rank $|S_0|$.

(ii) $\bold S$ is the lift of $\bold S_0$ to $W(k)$, i.e., ${\mathcal O}(\bold S)$ is a free $W(k)$-module of rank $|S_0|$. 
Thus, the tautological morphism $\pi: \bold G\to \bold G/\bold S$ is finite \'etale. 
\end{lemma} 

\begin{proof} 
(i) Let ${\rm Fun}(X,R)$ stand for the set of functions from $X$ to $R$.
We have a natural homomorphism $\tau: {\mathcal O}(\bold S)\to {\rm Fun}(S_0,W(k))$, given by 
$\tau(f)(g_0)=f(i(g_0))$. Since $\bold S_0$ is finite and reduced and there  is a lifting map $i: S_0\to S$,
the reduction of $\tau$ modulo $p^r$ is an isomorphism, implying (i).  

To prove (ii), note that $\tau$ is surjective, as it is surjective modulo $p$. 
Let $J:={\rm Ker}\tau$. Tensoring with $K$, we get a short exact sequence of $K$-vector spaces
$$
0\to K\otimes_{W(k)}J\to \mathcal O(\bold S_K)\to {\rm Fun}(S_0,K)\to 0
$$
Since $\bold S_K$ is automatically reduced, and $\bold S(K)=S$ is isomorphic to $S_0$ by Proposition \ref{ag2}, 
we obtain by counting dimensions that $K\otimes_{W(k)}J=0$, i.e., $J$ is the torsion ideal in ${\mathcal O}(\bold S)$. 
Since $\bold S$ is of finite type, $J$ is finitely generated, hence killed by $p^N$ for some $N\ge 1$. Hence, if $J\ne 0$, then there exists $f\in J$ such that 
$f\notin p{\mathcal O}(S)$. Since $p^Nf=0$, this contradicts (i) for $r>N$, yielding (ii). 
\end{proof} 

Consider the morphism $\phi: \bold G\to \bold V$ given by $\phi(g)=gv$, which we call {\it the orbit map}. It induces 
a natural morphism $\nu: \bold G/\bold S\to \bold V$ such that 
$\phi=\nu\circ \pi$. It is easy to see that every scheme-theoretic fiber of $\nu$ is either a point or empty. 
Hence, by  \cite{EGA}, Proposition 17.2.6, $\nu$ is a monomorphism. 

For a $W(k)$-scheme $\bold X$ and a closed point $x\in \bold X$ over $k$ or $\overline K$, denote by $\widehat{\bold X}_x$ the formal neighborhood of $\bold X$. 
Namely, if $R$ is a local Artinian $W(k)$-algebra with residue field $k$, respectively $\overline K$, then $\widehat{\bold X}_x(R)$ is the set of homomorphisms 
${\mathcal O}(\bold X)\to R$ which lift $x$. 

Let $\bold Y\subset \bold V$ is a closed $\bold G$-invariant subscheme such that $\nu$ factors through a morphism $\mu: \bold G/\bold S\to \bold Y$. 

\begin{proposition}\label{ag3} Suppose that 

(i) for any point $g$ of $\bold G/\bold S$ over $k$ or $\overline K$, 
the morphism of formal neighborhoods $\mu_g: \widehat{\bold G/\bold S}_g\to \widehat{\bold Y}_{\mu(g)}$ induced by $\mu$ 
is an isomorphism; and 

(ii)  $\bold Y$ consists of finitely many closed $\bold G$-orbits both over $k$ and over $\overline K$. 

Then $\mu$ is a closed embedding. In particular, the morphisms $\mu,\nu,\phi$ are finite.   
\end{proposition} 

\begin{proof} 
Let us cut down the target of $\mu$. Pick a $\bold G$-invariant polynomial $b_0\in \mathcal \O(\bold V_k)$ such that $b_0(v^0)=0$ 
and $b_0=1$ on all the other orbits of $\bold G$ in $\bold Y_k$. Since by (ii), $\bold Y_k$ consists of finitely many closed $\bold G$-orbits, such $b_0$ exists by Haboush's theorem, \cite{Ha}. 

Now use Proposition \ref{FKprop} to lift some power $b_0^N$ of $b_0$ to a $\bold G$-invariant polynomial $b\in {\mathcal O}(\bold V)$ such that $b(v)=0$ 
(namely, choose any lift $b$ of $b_0^N$ and then replace $b$ with $b-b(v)$). 
Also, consider a polynomial $c'\in {\mathcal O}(\bold V_K)$ such that $c'(v)=0$ but $c'\ne 0$ on all other orbits of $\bold G$ on $\bold Y_{\overline{K}}$. Since 
by (ii), $\bold Y_{\overline{K}}$ is a union of finitely many closed $\bold G$-orbits, $c'$ exists, and by setting $c'=p^Mc$ for sufficiently large $M\in \Bbb Z_+$, we obtain a polynomial $c\in {\mathcal O}(V)$. 

Now consider the closed subscheme $\bold Z\subset \bold Y$ cut out by the equations  $b=0, c=0$. Then the morphism
$\mu$ factors through a monomorphism $\bar\mu: \bold G/\bold S\to \bold Z$, 
and it suffices to show that this morphism is a closed embedding. 
We will do so by showing that, in fact, $\bar\mu$ is an isomorphism. 

\begin{lemma}\label{sur} $\bar\mu$ is surjective. 
\end{lemma} 

\begin{proof} This follows since by construction, $\bold Z$ has only one $\bold G$-orbit both over $k$ and over $\overline{K}$ 
(namely, that of $v^0$, respectively $v$).  
\end{proof} 

\begin{lemma}\label{eta} $\bar\mu$ is \'etale.
\end{lemma} 

\begin{proof} By (i), $\bar\mu$ is a formally \'etale morphism between affine $W(k)$-schemes of finite type. This implies the statement.  
\end{proof} 

Now the proposition follows from Lemma \ref{sur} and Lemma \ref{eta}, since a surjective \'etale monomorphism is an isomorphism,
\cite{EGA}, Theorem 17.9.1.
\end{proof} 
 
 \begin{remark}\label{finiteness} The reductivity of $\bold G$ is essential in Proposition \ref{ag3}. 
Namely, let $p>2$ and consider the action of $\bold G=\Bbb G_a$ by translations $x\mapsto x+b$ 
on the 3-dimensional space $\bold V$ of quadratic polynomials in $x$. Let $v=x-px^2$. 
Then $b\circ v=x+b-p(x+b)^2=b-pb^2+(1-2pb)x-px^2$. Thus, the map $\phi: \bold G\to \bold V$, $\phi(g)=gv$,
is given by $$\phi(b)=(b-pb^2,1-2pb,-p).$$ We have 
$$
\mathcal O(\phi^{-1}(0,-1,-p))=W(k)[b]/(b-pb^2,2-2pb)=W(k)[1/p]=K
$$ 
(as $pb=1$ in this ring). This implies that $\phi$ is not finite (as $K$ is not a finitely generated $W(k)$-module), even though it is finite over $\overline K$ and over $W(k)/I^N$ for each $N$.  

We note that in this example the orbits of $v$ over $\overline K$ and its reduction $v^0$ over $k$ are closed, and 
the scheme-theoretic stabilizers $\bold S$ of $v$ and $\bold S_0$ of $v^0$ are both trivial. Also, 
one may take $\bold Y$ to be the curve consisting of polynomials $-p x^2+\alpha x+\beta$ such that
$4p\beta=1-\alpha^2$. This curve is $\bold G$-invariant, and the map $\mu: \bold G=\bold G/\bold S\to \bold Y$ is a monomorphism which induces an isomorphism 
on formal neighborhoods. However, $\bold Y_k$ has two orbits, $\alpha=1$ and $\alpha=-1$, and $\mu$ lands in the first one. 
We have a $\bold G$-invariant polynomial $b_0$ on $\bold V_k$ separating these orbits, namely $b_0=(1-\alpha)/2$. 
But power reductivity does not apply since $\bold G$ is not reductive, and no power of $b_0$ lifts to a $\bold G$-invariant in ${\mathcal O}(\bold Y)$, since 
$\bold Y_{\overline K}$ has only one $\bold G$-orbit, so the only invariant regular functions on $\bold Y$ are constants. As a result, 
we cannot define a closed subscheme $\bold Z\subset \bold Y$ such that $\mu$ factors through a {\it surjective} morphism $\bar\mu: \bold G/\bold S\to \bold Z$.      
\end{remark}

\section{Faithfulness of the lifting for Hopf algebras} 

\subsection{Faithfulness of the lifting} 

If $H$ is a semisimple cosemisimple Hopf algebra over $k$, let $\widetilde H$ denote
its lift over $W(k)$ constructed in \cite{EG}, Theorem 2.1, and let  $\widehat H:=\overline{K}\otimes_{W(k)}\widetilde H$. 

Let $H_1,H_2$ be semisimple cosemisimple Hopf algebras over $k$. 

\begin{theorem}\label{t1}  (\cite{ENO}, Corollary 9.10) 
If $\widehat{H_1}$ is isomorphic to $\widehat{H_2}$ then $H_1$ is isomorphic to $H_2$. 
\end{theorem}  

\begin{proof}
By the assumption, $\dim H_1=\dim H_2=d$, a number coprime to $p$. Indeed, by Theorem 3.1 of \cite{EG} in a semisimple cosemisimple Hopf algebra we have $\Bbb S^2=1$, where $\Bbb S$ is the antipode, and by the Larson-Radford theorem (\cite{LR}), $\Tr(\Bbb S^2)\ne 0$. 

Fix identifications $\widetilde{H_1}\cong W(k)^d$, $\widetilde{H_2}\cong W(k)^d$ as $W(k)$-modules; they, in particular, define 
identifications $H_1\cong k^d$, $H_2\cong k^d$ as vector spaces over $k$.  

Let 
$\bold V$ be the space of all possible pre-Hopf structures, i.e., product, coproduct, unit, counit, 
and antipode maps on the $d$-dimensional space, without any axioms, regarded as an affine scheme. 
In other words, if $\bold E=\Bbb A^d$ is the defining representation of $\bold G$, then 
$$
\bold V=\bold E\bigotimes \bold E\bigotimes \bold E^*\oplus \bold E\bigotimes \bold E^*\bigotimes \bold E^*\oplus \bold E\oplus \bold E^*\oplus \bold E\bigotimes \bold E^*.
$$ 
Then $\bold V$ is a rational representation of $\bold G=GL(d)$, 
and $\widetilde{H_1},\widetilde{H_2}$ together with the above identifications 
give rise to two vectors $v_1,v_2\in \bold V(W(k))$, while $H_1,H_2$ correspond to their reductions mod $I$, $v_1^0,v_2^0\in \bold V(k)$. 
Moreover, by the assumption of the theorem, $v_1,v_2$ are conjugate under the action of $\bold G(\overline{K})$. 
We are going to show that the reductions $v_1^0,v_2^0$ are conjugate under the action of $\bold G(k)$, i.e. $H_1\cong H_2$, as claimed. 

Let $H\cong k^d$ be a semisimple cosemisimple Hopf algebra over $k$, and $u$ be the corresponding vector in $\bold V(k)$. 

\begin{lemma}\label{clo} (\cite{ENO}, Section 9) 
The $\bold G$-orbit $\bold Gu$ of $u$ is closed.
\end{lemma} 

\begin{proof} 
Let $u'\in \overline{\bold G u}$. 
Then $u'$ corresponds to a $d$-dimensional Hopf algebra $H'$ such that $\Tr|_{H'}(\Bbb S^2)=d\ne 0$, since this is so for all points of $\bold Gu$. Hence $H'$ is semisimple and cosemisimple by the 
Larson-Radford theorem (\cite{LR}). But then the stabilizers of $u,u'$ in $\bold G(k)$, which are isomorphic to ${\rm Aut}(H), {\rm Aut}(H')$, are finite, see \cite{EG}, Corollary 1.3. 
Hence, $\dim \bold Gu'=\dim \bold Gu=\dim \bold G$. This implies that $u'\in \bold Gu$. Hence, $\bold Gu$ is closed.   
\end{proof} 
 
Thus, the orbits of $v_1^0$, $v_2^0$ are closed. Hence, Theorem \ref{t1} follows from Proposition \ref{ag}. 
\end{proof} 

\subsection{Integrality of the stabilizer} 

\begin{theorem}\label{t2} Let $H_1,H_2$ be semisimple cosemisimple Hopf algebras over $k$, and 
$g: \widehat H_1\to \widehat H_2$ be an isomorphism. Then $g$ maps $\widetilde H_1$ isomorphically to $\widetilde H_2$, i.e., it is a lift of an isomorphism $g_0: H_1\to H_2$ over $W(k)$. 
In particular, for a semisimple cosemisimple Hopf algebra $H$ over $k$, 
the lifting map $i: {\rm Aut}(H)\to {\rm Aut}(\widehat{H})$ defined in \cite{EG}, Theorem 2.2
is an isomorphism.
\end{theorem} 

\begin{proof} By Theorem \ref{t1}, we may assume that $H_1=H_2=H$. 
Let $v\in \bold V(W(k))$ be the vector corresponding to $\widetilde H$, and $v^0\in \bold V(k)$ be its
reduction mod $I$, corresponding to $H$. Let $S:={\rm Aut}(\widehat{H})\subset G(\overline{K})$ be the stabilizer of $v$, 
$S_0:=\Aut(H)\subset \bold G(k)$ be the stabilizer of $v^0$, and $\bold S_0$ be the scheme-theoretic stabilizer of $v^0$. 

By Lemma \ref{clo}, the $\bold G$-orbit of $v^0$ is closed. Also, by \cite{EG}, Theorem 1.2, Corollary 1.3, $\bold S_0$ is finite and reduced. 
Finally, by \cite{EG}, Theorem 2.2, we have a lifting map $i: S_0\hookrightarrow S$ such that for any $g_0\in S_0$, $i(g_0)$ is integral, and the reduction modulo $p$ of $i(g_0)$ equals $g_0$. 
Thus, Proposition \ref{ag2}  applies, and result follows. 
\end{proof} 

\begin{remark} Theorems \ref{t1}, \ref{t2} also hold for quasitriangular and triangular Hopf algebras, with the same proofs. 
\end{remark} 

\subsection{Finiteness of the orbit map} 

We would now like to prove an analog of Lemma 9.7 of \cite{ENO} in the Hopf algebra setting. 

\begin{theorem}\label{t3} Let $H$ be a semisimple cosemisimple Hopf algebra over $k$, and $v\in \bold V(W(k))$ 
be the vector corresponding to $\widetilde{H}$. Then the morphism $\phi: \bold G\to \bold V$ defined by 
$\phi(g)=gv$ is finite. 
\end{theorem}  

\begin{proof} Let $\bold S$ be the scheme-theoretic stabilizer of $v$. Theorem \ref{t2} and Lemma \ref{finetale} imply 
that $\bold S$ is the lift of $\bold S_0$ over $W(k)$, and the tautological morphism $\pi: \bold G\to \bold G/\bold S$ is finite \'etale. We also have a natural morphism $\nu: \bold G/\bold S\to \bold V$ induced by $\phi$, such that 
$\phi=\nu\circ \pi$. 

Let $\bold Y\subset \bold V$ be the closed subscheme of Hopf algebra structures with $\Tr(\Bbb S^2)=d$.
Then $\bold Y$ is a closed $\bold G$-invariant subscheme, and $\nu$ factors through a morphism $\mu: \bold G/\bold S\to \bold Y$.
By the Larson-Radford theorem, \cite{LR}, $d\ne 0$ in $k$, so any Hopf algebra of dimension $d$ over $k$ or $\overline K$ is semisimple and cosemisimple. 
Hence, by Stefan's theorem \cite{St} (restated in \cite{EG}, Theorem 1.1), $\bold Y$ consists of finitely many orbits both over $k$ and over $\overline{K}$, 
which are closed by Lemma \ref{clo}. Also, it follows from \cite{EG}, Theorem 2.2, that $\mu$ induces an isomorphism on formal neighborhoods. 
Thus, Proposition \ref{ag3} applies, and the statement follows. 
\end{proof} 

\begin{remark} Theorems \ref{t1}, \ref{t2}, \ref{t3} are subsumed by the results of Section 4. 
However, we felt it was useful to give independent direct proofs of these theorems which do not use tensor categories. 
\end{remark} 

\subsection{Fullness of the lifting functor} 
 
Finally, let us prove the following result, which appears to be new. 

\begin{theorem}\label{t3a} Let $H_1,H_2$ be semisimple cosemisimple Hopf algebras over $k$. 
Then any Hopf algebra homomorphism $\theta: \widehat H_1\to \widehat H_2$ is a lifting 
of some homomorphism $\theta_0: H_1\to H_2$. In other words, the lifting functor 
defined by \cite{EG}, Corollary 2.4, is a fully faithful embedding from the category of semisimple cosemisimple Hopf algebras over $k$ 
to the category of semisimple (thus, cosemisimple) Hopf algebras over  $\overline{K}$. 
\end{theorem} 

\begin{proof} We first prove the following lemma. 

\begin{lemma}\label{bije} Let $H$ be a semisimple cosemisimple Hopf algebra over $k$. Then lifting 
defines a bijection between Hopf ideals of $H$ and Hopf ideals of $\widehat{H}$. 
The same applies to Hopf subalgebras. 
\end{lemma} 

\begin{proof} 
A Hopf ideal $J\subset A$ of a semisimple Hopf algebra $A$ corresponds to a full tensor subcategory $\C_J\subset \Rep H$ of objects annihilated by $J$, 
and this correspondence is a bijection. Full tensor subcategories, in turn, correspond to fusion subrings of the Grothendieck ring of $\Rep A$. 
So the first statement follows from the fact that the Grothendieck rings of $H$ and $\widehat{H}$ are the same. 
The second statement is dual to the first statement, since the orthogonal complement of a Hopf ideal in $A$ is a 
Hopf subalgebra of $A^*$, and vice versa.  
\end{proof} 

Now let $B={\rm Im}\theta\subset \widehat{H}_2$. Then by Lemma \ref{bije}, $B$ is a lifting of some Hopf subalgebra $B_0\subset H_2$. 
Note that $B_0$ is semisimple, since its dimension divides the dimension of $H_2$ by the Nichols-Zoeller theorem (\cite{NZ}), hence is coprime to $p$. 
Thus, without loss of generality we may replace $H_2$ with $B_0$, i.e., assume that $\theta$ is surjective. 

Now let $J={\rm Ker}\theta\subset \widehat H_1$. Then by Lemma \ref{bije}, $J$ is a lift of some Hopf ideal $J_0\subset H_1$. 
Let $C_0=H_1/J_0$. Then $C_0$ is cosemisimple, as it is a quotient of $H_1$, so by the Nichols-Zoeller theorem,  
its dimension is coprime to $p$. Moreover, $\widehat{C}\cong \widehat{H}_1/J$. Thus, without loss of generality we may replace $H_1$ with $C_0$, i.e., assume that $\theta$ is an isomorphism. 

But in this case the desired statement is Theorem \ref{t2}. 
\end{proof} 

\section{Faithfulness of the lifting for fusion categories} 

\subsection{Faithfulness of the lifting} 

Now we generalize the results of the previous section to separable fusion categories. 
(For basics on tensor categories we refer the reader to \cite{EGNO}.)
Most of the proofs are parallel to the Hopf algebra case, and we will indicate the necessary modifications. 
We will develop the theory for ordinary fusion categories; the case of braided and symmetric categories is completely parallel. 

We call a fusion category $\C$ {\it separable} if its global dimension is nonzero. This is equivalent to the definition of \cite{DSS} by \cite{DSS}, Theorem 3.6.7. 

If $\C$ is a (braided, symmetric) separable fusion category over $k$, let $\widetilde \C$
be its lift to $W(k)$ constructed in \cite{ENO}, Theorem 9.3, Corollary 9.4, and let  $\widehat \C:=\overline{K}\otimes_{W(k)}\widetilde \C$. 

Let $\C_1,\C_2$ be (braided, symmetric) separable fusion categories over $k$.
First we prove the following theorem, which is Corollary 9.9(i) of \cite{ENO}.

\begin{theorem}\label{t4}  
If $\widehat{\C_1}$ is equivalent to $\widehat{\C_2}$ then $\C_1$ is equivalent to $\C_2$. 
\end{theorem}  

\begin{proof} We will treat the fusion case; the braided and symmetric cases are similar. 

We generalize the proof of Theorem \ref{t1}. As in \cite{ENO}, Subsection 9.3, we may assume that $\C_1$ and $\C_2$ have the same underlying 
semisimple abelian category $\overline\C$ with the tensor product functor $\otimes$, a skeletal category with Grothendieck ring ${\rm Gr}(\overline\C)$. 
So it has simple objects $X_i$, $i\in \Bbb I$, with $X_0=\bold 1$, and $X_i\otimes X_j=\oplus_m k^{N_{ij}^m}X_m$. 
We will also fix the unit morphism $\iota: \bold 1\otimes \bold 1\to \bold 1$, and the coevaluation morphisms. 
Define a pre-tensor structure on $\overline\C$ to be a triple $(\Phi,\Phi',{\rm ev})$, where $\Phi$ is an associativity morphism, 
$\Phi'$ an "inverse" associativity morphism in the opposite direction, and ${\rm ev}$ is a collection of evaluation morphisms, 
but without any axioms. Then a tensor structure on $\overline\C$ is a pre-tensor structure such that $\Phi\circ \Phi'=\Phi'\circ \Phi={\rm Id}$ 
and $(\Phi,{\rm ev})$ satisfy the axioms of a rigid tensor category (with the fixed unit and coevaluation morphisms), see \cite{EGNO}, Definitions 2.1.1 and 2.10.11. 

Let
$$
N_{ijl}^s:=[X_i\otimes X_j\otimes X_l:X_s]=\sum_m N_{ij}^mN_{ml}^s=\sum_p N_{ip}^sN_{jl}^p.
$$ 
Let $\bold V$ be the space of all pre-tensor structures on $\overline\C$, an affine space over $W(k)$ of dimension
$2\sum_{i,j,l,s} (N_{ijl}^s)^2+{\rm rank}{\rm Gr}(\overline\C)$ (where the first summand corresponds to pairs $(\Phi,\Phi')$ and the second summand to ${\rm ev}$).   
Then $\widehat{\C}_i$ give rise to vectors $v_i\in \bold V(W(k))$, and $\C_i$ correspond to their reductions $v_i^0$ modulo $I$. 

Now let us define the relevant affine group scheme $\bold G$. To this end, following \cite{ENO}, Subsection 9.3, let $\bold T=\Aut(\otimes)$ 
be the group of automorphisms of the tensor product functor on $\overline\C$. Then $\bold T$ acts naturally on $\bold V$ by twisting (where after twisting 
we renormalize the coevaluation and unit morphisms to be the fixed ones). Let $\bold T_0\subset \bold T$ be the subgroup of ``trivial twists", 
i.e. ones of the form $J_{X,Y}=z_{X\otimes Y}(z_X^{-1}\otimes z_Y^{-1})$, where $z\in \Aut(\Id_{\overline\C})$ is an invertible element of the center of $\overline\C$. 
Then $\bold T_0$ is a closed central subgroup of $\bold T$ which acts trivially on $\bold V$. Let $\bold G:=\bold T/\bold T_0$. Then $\bold G$ is an affine group scheme acting rationally on $\bold V$. Moreover, $\bold T=\prod_{i,j,m}GL(N_{ij}^m)$, and $\bold T_0$ is a central torus in $\bold T$, hence $\bold G$ is a split reductive group. 

Let $\C$ be a separable fusion category over $k$ of some global dimension $d\ne 0$, and $u$ be the corresponding vector in $\bold V(k)$. 
Let $\Aut(\C)$ denote the group of isomorphism classes of tensor autoequivalences 
of $\C$.

\begin{lemma}\label{clo1} (\cite{ENO}, Section 9) 
The $\bold G$-orbit $\bold Gu$ of $u$ is closed.
\end{lemma} 

\begin{proof} 
Let $u'\in \overline{\bold G u}$. 
Then $u'$ corresponds to a fusion category $\C'$. Moreover the global dimension of $\C'$ is $d$, 
since this is so for all points of $\bold Gu$, and the global dimension depends algebraically on $\Phi,\Phi',{\rm ev}$. 
Thus, $\C'$ is separable. But then the stabilizers of $u,u'$ in $\bold G(k)$, which are isomorphic to $\Aut(\C)$, ${\rm Aut}(\C')$ are finite by \cite{ENO}, Theorem 2.31
(which applies in characteristic $p$ for separable categories, see \cite{ENO}, Section 9).  
Hence, $\dim \bold Gu'=\dim \bold Gu=\dim \bold G$. This implies that $u'\in \bold Gu$. Hence, $\bold Gu$ is closed.   
\end{proof} 

Thus, the orbits of $v_1^0$, $v_2^0$ are closed. Hence, Theorem \ref{t4} follows from Proposition \ref{ag} and the fact 
that the natural map $\bold T(W(k))\to (\bold T/\bold T_0)(W(k))$ is surjective. 
\end{proof} 

\begin{remark} Note that Theorem \ref{t4} implies \cite{ENO}, Theorem 9.6; namely, the complete local ring $R$ 
in Theorem 9.6 without loss of generality may be replaced by $W(k)$. 
\end{remark} 

\begin{remark}\label{egcorr} 
We can now complete the proof of \cite{EG2}, Theorem 6.1, which states, essentially, that 
any semisimple cosemisimple triangular Hopf algebra over $k$ is a twist of a group algebra. 
The original proof of this theorem which appeared in \cite{EG2} is incomplete 
(namely, it is not clear at the end of this proof why $F\circ F'={\rm Id}$). 
This is really a consequence of faithfulness of the lifting. Namely, if $(A,R)$ 
is a semisimple cosemisimple triangular Hopf algebra over $k$, and $(A',R')=F\circ F'(A,R)$, 
then $(A,R)$ and $(A',R')$ have isomorphic liftings over $\overline K$, hence by Theorem \ref{t1}
they are isomorphic. 

Another proof of \cite{EG2}, Theorem 6.1 is obtained by using Theorem \ref{t4} for symmetric tensor categories. 
Namely, consider the separable symmetric fusion category $\C:=\Rep (A,R)$.
Then $\widehat \C$ is a symmetric fusion category over $\overline K$. Hence, by Deligne's theorem (\cite{D}, see also \cite{EGNO}, Corollary 9.9.25), 
$\widehat \C=\Rep_{\overline K}(G,z)$, where $G$ is a finite group of order coprime to $p$ 
and $z\in G$ is a central element of order $\le 2$ (the category of representations of $G$ on superspaces with parity defined by $z$). 
Thus, $\widehat \C\cong \widehat \D$ as symmetric tensor categories, where $\D=\Rep_k(G,z)$. 
Hence, by Theorem \ref{t4}, $\C\cong \D$, i.e., $(A,R)$ is obtained by twisting of $(k[G],R_z)$, 
where $R_z=1\otimes 1$ if $z=1$ and $R_z=\frac{1}{2}(1\otimes 1+1\otimes z+z\otimes 1-z\otimes z)$ 
if $z\ne 1$ (note that if $z\ne 1$ then $|G|$ is necessarily even, so $p>2$ and $1/2\in k$). 
\end{remark} 

\begin{remark}\label{nik} Let $G$ be a finite group. Recall (\cite{EGNO}) 
that categorifications of the group ring $\Bbb ZG$ over a field $F$ correspond to elements of $H^3(G,F^\times)$. 
Hence, the lifting map for pointed fusion categories which categorify $\Bbb ZG$ is the natural map $\alpha: H^3(G,k^\times)\to H^3(G,\overline K^\times)$ arising from the isomorphisms $H^i(G,\overline{K}^\times)\cong H^i(G,\overline{K}^\times_f)$ and $H^i(G,k^\times)\cong H^i(G,k^\times_f)$, where $\overline{K}^\times_f$ is the group of elements of finite order (i.e., roots of unity) in $\overline{K}^\times$, and $k^\times_f$ is the group of roots of unity in $k^\times$. 
The map $\alpha$ is then induced by the Brauer lift map $k^\times_f\to \overline{K}^\times_f$, and is clearly injective since the Brauer lift identifies $k^\times_f$ with a direct summand of $\overline{K}^\times_f$. Thus, for pointed categories faithfulness of the lifting is elementary. 

Similarly, braided categorifications $\Bbb ZG$ over $F$ (for abelian $G$) correspond to quadratic forms $G\to F^\times$ (see \cite{EGNO}), and lifting for such categorifications is defined by the Brauer lift for quadratic forms, hence is clearly faithful.  
\end{remark} 

\subsection{Integrality of the stabilizer} 

\begin{theorem}\label{t5} Let $\C_1,\C_2$ be (braided, symmetric) separable fusion categories over $k$, and 
$g: \widehat \C_1\to \widehat \C_2$ be an equivalence. Then $g$ defines an equivalence $\widetilde \C_1\to \widetilde \C_2$, i.e., it is isomorphic to the lift of an equivalence $g_0: \C_1\to \C_2$ over $W(k)$. 
\end{theorem} 

\begin{proof} As before, we treat only the fusion case; the braided and symmetric cases are similar. 
By Theorem \ref{t4}, we may assume that $\C_1=\C_2=\C$. 
Let $v\in \bold V(W(k))$ be the vector corresponding to $\widetilde \C$, and $v^0\in \bold V(k)$ be its
reduction mod $I$, corresponding to $\C$. Let $S:={\rm Aut}(\widehat{\C})$ be the stabilizer of $v$, 
$S_0:={\rm Aut}(\C)$ be the stabilizer of $v^0$, and $\bold S_0$ the scheme-theoretic stabilizer of $v^0$. 

By Lemma \ref{clo1}, the $\bold G$-orbit of $v^0$ is closed. 
By Theorem 2.27 and Theorem 2.31 of \cite{ENO} (both valid in characteristic $p$ 
for separable fusion categories, see \cite{ENO}, Section 9), 
$\bold S_0$ is finite and reduced. Finally, by \cite{ENO}, Theorems 9.3, 9.4, 
we have a lifting map $i: S_0\hookrightarrow S$, such that for all $g_0\in \bold G(k)$, $i(g_0)$ is integral and 
the reduction of $i(g_0)$ modulo $p$ equals $g_0$. 
Thus, Proposition \ref{ag2}  applies, and the result follows
(again using that the natural map $\bold T(W(k))\to (\bold T/\bold T_0)(W(k))$ is surjective). 
\end{proof} 

\begin{remark} 1. Recall that a multifusion category is called separable if all of its component fusion categories are separable, see \cite{ENO}, Section 9 and \cite{DSS}. Theorems \ref{t7} and \ref{t8} extend to separable multifusion categories with similar proofs.  

2. Theorem \ref{t5} is not stated explicitly in \cite{ENO}, but is claimed implicitly in the (incomplete) proof of \cite{ENO}, Theorem 9.6.  
\end{remark} 

\subsection{Finiteness of the orbit map} 

Let us now prove Lemma 9.7 of \cite{ENO} (which completes the proofs in \cite{ENO}, Subsection 9.3). 

\begin{theorem}\label{t6} Let $\C$ be a (braided,symmetric) separable fusion category over $k$, and $v\in \bold V(W(k))$ 
be the vector corresponding to $\widetilde{\C}$. Then the morphism $\phi: \bold G\to \bold V$ defined by 
$\phi(g)=gv$ is finite. 
\end{theorem}  

\begin{proof} We treat the case of fusion categories; the braided and symmetric cases are similar. 
The proof is parallel to the proof of Theorem \ref{t3}. Namely, let $\bold S$ be the scheme-theoretic stabilizer of $v$.
Then $\phi=\nu\circ \pi$, where $\pi: \bold G\to \bold G/\bold S$ is finite \'etale, and $\nu: \bold G/\bold S\to \bold V$. 
Let $\bold Y\subset \bold V$ denote the closed subscheme of vectors corresponding 
to fusion categories of global dimension $\widetilde d:=\dim(\widetilde\C)$. Then $\nu$ factors through 
$\mu: \bold G/\bold S\to \bold Y$. By \cite{ENO}, Theorem 2.27, $\bold Y$ consists of finitely many 
$\bold G$-orbits both over $k$ and over $\overline K$. Also, these orbits are closed by Lemma \ref{clo1}. 
Finally, by \cite{ENO}, Theorem 9.3, Corollary 9.4, $\mu$ induces an isomorphism on formal neighborhoods. 
Thus, Proposition \ref{ag3} applies, and the statement follows.  
\end{proof} 

\subsection{Faithfulness of the lifting and integrality of the stabilizer for tensor functors} 

Let us now prove similar results for tensor functors, i.e., Theorem 9.8  and Corollary 9.9(ii) of \cite{ENO}. 

\begin{theorem}\label{t7}
Let $\C,\D$ be two (braided, symmetric) separable fusion categories over $k$, and 
$F_1,F_2:\C\to \D$ two (braided) tensor functors. 
Let $g: \widehat F_1\to \widehat F_2$ be an isomorphism of lifts of $F_1,F_2$ over $\overline{K}$. Then $g$ 
is a lift of an isomorphism $g_0: F_1\to F_2$. In particular, if $\widehat F_1$, $\widehat F_2$ 
are isomorphic then so are $F_1,F_2$.  
\end{theorem} 

\begin{proof} The proof is similar to the proofs of Theorems \ref{t1}, \ref{t2} and Theorems \ref{t4}, \ref{t5}. 
We treat the case of tensor functors between fusion categories; the cases of braided and symmetric categories 
are similar. 

We may assume that $F_1,F_2$ coincide with a given functor $\overline F$ as additive functors, and differ only by the tensor structures.
Let $\bold V$ be the space of pre-tensor structures on $\overline F$, i.e., pairs $(J,J')$ of endomorphisms of the functor $\overline F(\cdot\otimes \cdot)$, without any axioms.
Then a tensor structure on $\overline F$ is such a pair $(J,J')$ for which $J\circ J'=J'\circ J={\rm id}$, and $J$ satisfies the tensor structure axiom, \cite{EGNO}, Definition 2.4.1. 
Let $\bold G$ be the group scheme of all automorphisms of the functor $\overline F$. Then $\bold G$ is a split reductive group (a product of general linear groups)
which acts on $\bold V$ by ``gauge transformations".  Moreover, the functors $\widetilde F_j$, $j=1,2$, correspond to vectors $v_j\in \bold V(W(k))$, and $F_j$ correspond to their reductions $v_j^0$ modulo $p$. 

\begin{lemma}\label{clo2} Let $u\in \bold V(k)$ be a vector representing a tensor functor $F$.
Then the orbit $\bold Gu$ is closed.
\end{lemma} 

\begin{proof} 
Let $u'\in \overline{\bold G u}$. 
Then $u'$ corresponds to a tensor functor $F'$. 
But the group of automorphisms of a tensor fuctor between separable fusion categories is finite by \cite{ENO}, Theorem 2.27. 
Hence the stabilizers of $u,u'$ in $\bold G(k)$, which are isomorphic to ${\rm Aut}(F)$, ${\rm Aut}(F')$, are finite.  
Hence, $\dim \bold Gu'=\dim \bold Gu=\dim \bold G$. This implies that $u'\in \bold Gu$. Hence, $\bold Gu$ is closed.   
\end{proof} 

By Lemma \ref{clo2}, Proposition \ref{ag} applies. Thus, $F_1\cong F_2$ as tensor functors. So we may assume without loss of generality 
that $F_1=F_2=F$ for some tensor functor $F$. 

Let $v\in \bold V(W(k))$ and $v^0\in \bold V(k)$ be its reduction modulo $p$. Let $S_0=\Aut(F)\subset \bold G(k)$, $\bold S_0$ be the scheme-theoretic stabilizer of $v^0$, 
and $S={\rm Aut}(\widehat F)\subset \bold G(\overline K)$. 
Then $\bold S_0$ is finite and reduced by \cite{ENO}, Theorem 2.27.
Also, by \cite{ENO}, Theorem 9.3 and Corollary 9.4, we have a lifting map $i: S_0\hookrightarrow S$
such that  for any $g_0\in S_0$, $i(g_0)$ is integral and the reduction of $i(g_0)$ modulo $p$ is $g_0$. 
Hence, Proposition \ref{ag2}  applies, and the result follows. 
\end{proof}

\begin{corollary}\label{t8}
For a (braided, symmetric) separable fusion category $\C$ over $k$, 
the lifting map $i: {\rm Aut}(\C)\to {\rm Aut}(\widehat{\C})$ defined in \cite{ENO}, Theorem 9.3, 
is an isomorphism. 
\end{corollary} 

\begin{proof} Theorem \ref{t7} implies that $i$ is injective, and Theorem \ref{t5} implies that $i$ is surjective. 
\end{proof} 

\subsection{Application to Brauer-Picard and Picard groups of tensor categories} 

Recall from \cite{ENO2} that to any fusion category $\C$ one can attach its Brauer-Picard groupoid
$\underline{\underline{\rm BrPic}}(\C)$. This is a 3-group, whose 1-morphisms are 
equivalence classes of invertible $\C$-bimodule categories, 2-morphisms are 
bimodule equivalences of such bimodule categories, and 3-morphisms are 
isomorphisms of such equivalences. Similarly, if $\C$ is braided then one can define its Picard groupoid 
$\underline{\underline{\rm Pic}}(\C)$, a 3-group whose 1-morphisms are equivalence classes of invertible $\C$-module categories,
2-morphisms are module equivalences of such module categories, and 3-morphisms are 
isomorphisms of such equivalences. 

Theorems 9.3, 9.4 of \cite{ENO} then imply that if $\C$ is separable then we have 
the lifting morphism  $i: \underline{\rm BrPic}(\C)\to \underline{\rm BrPic}(\widehat \C)$
and (in the braided case) $i: \underline{\rm Pic}(\C)\to \underline{\rm Pic}(\widehat \C)$
of the underlying 2-groups. 

\begin{corollary}\label{t9} Let $\C$ be a separable fusion category over $k$. 

(i) The lifting morphism $i: \underline{\rm BrPic}(\C)\to \underline{\rm BrPic}(\widehat \C)$ is an isomorphism;

(ii) If $\C$ is braided then the lifting morphism $i: \underline{\rm Pic}(\C)\to \underline{\rm Pic}(\widehat \C)$ is an isomorphism.
\end{corollary} 

\begin{proof} (i) By \cite{ENO2}, Theorem 1.1, for any separable fusion category $\D$ one has an isomorphism 
$\xi: {\rm BrPic}(\D)\cong \Aut^{\rm br}(\mathcal{Z}(\D))$ of the Brauer-Picard group ${\rm BrPic}(\D)$ with the group of isomorphism classes of braided autoequivalences of the Drinfeld center of $\D$. It is clear that this isomorphism is compatible with lifting. Therefore, the statement at the level of 1-morphisms follows from Corollary \ref{t8}. 
Also, recall that $\pi_2(\underline{\rm BrPic}(\D))={\rm Inv}(\mathcal Z(\D))$, the group of isomorphism 
classes of invertible objects of $\mathcal{Z}(\D)$. Thus, at the level of 2-morphisms $i$ comes from the obvious isomorphism 
${\rm Inv}(\mathcal{Z}(\C))\cong {\rm Inv}(\mathcal{Z}(\widehat \C))$, which gives (i).

(ii)  By \cite{DN}, if $\D$ is braided then ${\rm Pic}(\D)$ is naturally identified with the subgroup of $\Aut^{\rm br}(\mathcal{Z}(\D))$ of elements 
that preserve $\D\subset {\mathcal Z}(\D)$ and have trivial restriction to ${\mathcal D}$. Thus, (ii) follows from (i) and 
Theorem \ref{t7}. Also, recall that $\pi_2(\underline{\rm Pic}(\D))={\rm Inv}(\D)$, the group of isomorphism 
classes of invertible objects of $\D$. Thus, at the level of 2-morphisms $i$ comes from the obvious isomorphism 
${\rm Inv}(\C)\cong {\rm Inv}(\widehat \C)$, which gives (ii). 
\end{proof} 

Note that in Corollary \ref{t9}, $i$ does not define an isomorphism of $3$-groups, since $\pi_3$ of these 3-groups is the multiplicative group of the ground field, 
and $k^\times \ncong \overline K^\times$. However, we can lift $i$ to an injection at the level of 3-morphisms. For simplicity assume that $k=\overline {\Bbb F}_p$ 
(this is not restrictive since by \cite{ENO}, Theorem 2.31, any separable fusion category in characteristic $p$ is defined over $\overline {\Bbb F}_p$). 
Then by Hensel's lemma, the surjection $W(k)^\times \to k^\times$ defined by reduction modulo $p$ uniquely splits, since all elements of $k^\times$ are roots of unity of order coprime to $p$ (the Brauer lift, cf. Remark \ref{nik}). Then $i$ extends to a morphism of $3$-groups using the corresponding splitting $\beta: k^\times \to W(k)^\times \subset \overline K^\times$. 
In particular, using the main results of \cite{ENO2}, this implies the following result. 

\begin{theorem}\label{t11} Let $G$ be a finite group and $k=\overline{\Bbb F}_p$. 
Then any \linebreak $G$-extension of $\C$ canonically lifts to a $G$-extension of $\widehat \C$, and 
any braided $G$-crossed category $\D$ with $\D_1=\C$ canonically lifts to a braided \linebreak $G$-crossed category 
$\widehat \D$ with $\widehat \D_1=\widehat \C$.  
\end{theorem}

\begin{remark} One can propose the following definition (which we are not making completely precise here). 
Recall that a  {\it linear algebraic structure} is defined by a colored PROP $\bold P$ (say, over $\Bbb Z$), see e.g. \cite{YJ} and references therein. 
Realizations of $\bold P$ over a commutative ring $R$ are then $\bold P$-algebras over $R$. 
We call an algebraic structure $\bold P$ {\it 3-separable} 
if every finite-dimensional realization $A$ of $\bold P$ over a field $k$ 

(i) has a finite and reduced group of automorphisms (i.e., has no nontrivial derivations); 

(ii) has no nontrivial first order deformations; 

(ii) has a vanishing space of obstructions to deformations.

These conditions should be expressed as requiring that $H^i(A)=0$ for $i=1,2,3$ 
for an appropriate cohomology theory, controlling deformations of $A$. 
Then $A$ should admit a unique lifting from a field $k$ of characteristic $p$ to $W(k)$, 
and this lifting should have the faithfulness and stabilizer 
integrality properties similar to Theorems \ref{t1}, \ref{t2}, \ref{t4}, \ref{t5}, \ref{t7}: any isomorphism $g$ of 
liftings of $A_1,A_2$ over $\overline K$ is defined over $W(k)$ and hence is a  lifting of an isomorphism $g_0: A_1\to A_2$.    
We have seen a number of examples of 3-separable structures: 
semisimple cosemisimple (quasitriangular, triangular) Hopf algebras, separable (braided, symmetric) fusion categories, 
(braided) tensor functors between such categories.\footnote{Semisimplicity/cosemisimplicity for Hopf algebras and separability for fusion categories may be forced
in the setting of linear algebraic structures by adding an auxiliary variable $x$ and the relation $dx=1$, where $d$ is the global dimension.} 
It would therefore be interesting to make this notion more precise, and prove a general theorem 
on the existence and faithfulness of the lifting for 3-separable structures, which would unify the results of \cite{ENO}, Section 9, \cite{EG}, and this paper. 
It would also be interesting to find other examples of 3-separable structures. 
\end{remark} 

\section{Descent of tensor functors between separable fusion categories to characteristic $p$.}

\subsection{Separability of subcategories and quotient categories}

We will first prove separability of subcategories and quotients of separable categories.
First we need the following result, which is a generalization of 
\cite{EO}, Theorem 2.5. 

\begin{theorem}\label{surjexac} Let $\C$ be a finite tensor category 
and $\D$ a finite indecomposable multitensor category. Let $F:\C\to \D$ be a quasi-tensor functor. 
If $F$ is surjective then 

(i) $F$ maps projective objects to projective ones; and  
 
(ii) $\D$ is an exact module category over $\C$. 
\end{theorem} 

\begin{proof} (i) The proof is almost identical to the proof of \cite[Theorem 2.9]{EG3}. 
We reproduce it here for the convenience of the reader. 

Let $P_i$ be the indecomposable projectives of $\C$. Write $F(P_i)=T_i\oplus N_i$, where $T_i$ is projective, and $N_i$ has no projective direct summands. 
Our job is to show that $N_i=0$ for all $i$. So let us assume for the sake of contradiction 
that $N_p\ne 0$ for some $p$. 

Let $P_i\otimes P_j=\oplus_r c_{ij}^r P_r$. Since the tensor product of a projective object with any object in $\D$ is projective, the objects $T_i\otimes T_j$, $T_i\otimes N_j$, and $N_i\otimes T_j$ are projective.
Thus, 
\begin{equation}\label{incluu}
(\oplus_i N_i)\otimes N_j\supset \oplus_r (\sum_i c_{ij}^r) N_r
\end{equation}
 as a direct summand. 
 
 Let $\bold 1=\sum_l \bold 1_l$ be the irreducible decomposition of the unit object of $\D$, 
and let $s$ be such that $Y:=\sum_r d_rN_r\bold 1_s$ is nonzero (it exists since $N_p\ne 0$).
Also, let $X_j$ be the simples of $\C$, and $d_j$ be their Frobenius-Perron dimensions. 
For any $i$, $r$, we have $\sum_j d_jc_{ij}^r=D_id_r$, where $D_i:=\FPdim(P_i)$. 
Thus, tensoring inclusion \eqref{incluu} on the right by $\bold 1_s$, 
multiplying by $d_j$ and summing over $j$,
we get a coefficient-wise inequality  
$$
[\oplus_i N_i][Y]\ge (\sum_i D_i)[Y]
$$
in ${\rm Gr}(\D_s)$, where $\D_s:=\D \bold \otimes 1_s$. 
This implies that the largest eigenvalue of the matrix of $[\oplus_i N_i]$ on ${\rm Gr}(\D_s)$
is at least $\sum_i D_i$, which is the same as the largest eigenvalue of $[\oplus_i F(P_i)]$.
Since $F$ is surjective, all the entries of $[\oplus_i F(P_i)]$ are positive. 
Thus, by the Frobenius-Perron theorem (see Lemma 2.1 of \cite{EG3}), $[\oplus_i N_i]=[\oplus_i F(P_i)]$. This implies that $N_i=F(P_i)$ for all $i$. Thus, $F(P_i)$ has no nonzero projective direct
summands for all $i$. 

However, let $Q$ be an indecomposable projective object in $\D$. 
Then $Q$ is injective by the quasi-Frobenius property of finite multitensor categories. 
Since $F$ is surjective, $Q$ is a subquotient, hence a direct summand 
of $F(P)$ for some projective $P\in \C$. Hence $Q$ is a direct summand 
of $F(P_i)$ for some $i$, which gives the desired contradiction. 

(ii) This follows from (i) and the fact that in a multitensor category, the tensor product of a projective object with any object is projective. 
\end{proof}

\begin{theorem}\label{quo} Let $\C$, $\D$ be multitensor categories and let
$F: \C\to \D$ be a surjective tensor functor. If $\C$ is separable, then so is $\D$. In other words, a quotient category of a separable multifusion category is separable. 
\end{theorem} 

\begin{proof} Consider first the special case when $\C$ is a tensor category (i.e., $\bold 1\in \C$ is simple). Without loss of generality, we may assume that $\D$ is indecomposable. By Theorem \ref{surjexac}, $\D$ is an exact $\C$-module category, hence semisimple (as $\C$ is semisimple). 
Moreover, we have the surjective tensor functor
$F\boxtimes F: \C\boxtimes \C^{\rm op}\to \D\boxtimes \D^{\rm op}.$
Let us take the dual of this functor with respect to $\D$. By \cite{ENO}, Proposition 5.3, we get
an injective tensor functor (i.e., a fully faithful embedding) $$(F\boxtimes F)^*_\D:  \mathcal{Z}(\D)\hookrightarrow 
(\C\boxtimes \C^{\rm op})_\D^*.$$ But the category $(\C\boxtimes \C^{\rm op})_\D^*$ is semisimple, since $\C$ is separable and $\D$ is semisimple. Hence, ${\mathcal Z}(\D)$ is semisimple. Thus, by Corollary 3.5.9 of \cite{DSS}, $\D$ is separable.

The general case now follows by applying the above special case to the surjective tensor functors 
$F_i: \C_{ii}\to F(\bold 1_i)\otimes \D\otimes F(\bold 1_i)$, where $\bold 1_i$ are the simple composition factors of $\bold 1$ in $\C$, and $\C_{ii}:=\bold 1_i\otimes \C\otimes \bold 1_i$. 
\end{proof} 

The following theorem resolves an open question in \cite{ENO}, Subsection 9.4: 

\begin{theorem}\label{subca} A (full) multitensor subcategory of a separable multifusion category is separable.
\end{theorem} 

\begin{proof} Let $F: \C\hookrightarrow \D$ be an inclusion of $\C$ into a separable category $\D$. Then by \cite{ENO}, 
Proposition 5.3, we have a surjective tensor functor $F^*_\D: \D\to \C_\D^*$, where $\C_\D^*$ is a multitensor category. By Theorem \ref{quo}, $\C_\D^*$ is separable. Hence $\D=(\C_\D^*)_\D^*$ is separable as well. 
\end{proof} 

\subsection{Descent of tensor functors} 

\begin{theorem}\label{t7a}
Let $\C_1,\C_2,\D$ be separable multifusion categories over $k$. Let 
$F_i: \C_i\to \D$ be tensor functors. 
Let $F: \widehat{\C_1}\to \widehat{\C_2}$ 
be a tensor functor such that 
$\widehat F_2\circ F\cong \widehat F_1$. Then 
there exists a tensor functor 
$F_0: \C_1\to \C_2$ such that $F_2\circ F_0\cong F_1$, 
and $F\cong \widehat F_0$. In other words, if tensor functors $G$ and $G\circ F$ are integral then $F$ is integral. 
\end{theorem} 

Note that Theorem \ref{t3a} is recovered from Theorem \ref{t7a} when $\D=\Vec_k$ (namely, $H_i={\rm Coend}F_i$). 
Moreover, if $\D=\Rep A$, where $A$ is a semisimple $k$-algebra, then 
Theorem \ref{t7a} gives a generalization of Theorem \ref{t3a} to weak Hopf algebras. 

\begin{proof} The proof is similar to the proof of Theorem \ref{t3a}. 
If $\C$ is a separable multifusion category over $k$, there is a natural bijection 
between full tensor subcategories of $\C$ and $\widehat{\C}$. 
In particular, ${\rm Im}F$ is a lift of some multifusion subcategory 
$\mathcal E\subset \C_2$. 
By Theorem \ref{subca}, $\mathcal E$ is separable. Thus, 
we may replace $\C_2$ by $\mathcal E$, i.e., we may assume without loss of generality that $F$ is surjective. 

Also, we may replace $\D$ with the image ${\rm Im}F_2$ of $F_2$, which is separable by Theorem \ref{quo} (or Theorem \ref{subca}), i.e., we may assume without loss of generality that $F_2$ (hence $\widehat{F_2}$, hence $\widehat{F_1}$, 
hence $F_1$) is surjective. 

Consider the dual functor $F^*_{\widehat\D}: (\widehat{\C}_2)_{\widehat{\D}}^*\to (\widehat{\C}_1)_{\widehat{\D}}^*=\widehat{(\C_1)^*_{\D}}$, which is an inclusion of multifusion categories by \cite{ENO}, Proposition 5.3. Thus, $(\widehat{\C}_2)_{\widehat{\D}}^*$ is a lift of some multifusion subcategory 
$\B$ of $({\C}_1)_{\D}^*$. Moreover, by Theorem \ref{subca}, $\B$ is separable, so $(\widehat{\C}_2)_{\widehat{\D}}^*=\widehat{(\C_2)_\D^*}\cong \widehat \B$. 
Let $H:  \B\hookrightarrow ({\C}_1)_{\D}^*$ be the corresponding inclusion functor. 
Then the dual functor $H^*_{\D}: \C_1\to \B^*_{\D}$ is a surjection, and 
$F\cong F'\circ \widehat{H^*_{\D}}$, where $F': \widehat{\B^*_{\D}}\cong \widehat{\C_2}$ 
is an equivalence. By Theorem \ref{t5}, $F'$ is isomorphic to the lift $\widehat{F_0'}$ of an equivalence $F_0': {\B^*_{\D}}\cong {\C_2}$, hence $F$ is isomorphic to the lift $\widehat{F_0}$ of a tensor functor $F_0=F_0'\circ H^*_{\D}: \C_1\to \C_2$. 
This proves the theorem. 
\end{proof} 

\begin{theorem} \label{t7b}
Let $\C_1,\C_2,\D$ be separable multifusion categories over $k$. Let 
$F_i: \D\to \C_i$ be tensor functors, such that $F_1$ is surjective. 
Let $F: \widehat{\C_1}\to \widehat{\C_2}$ 
be a tensor functor such that 
$F\circ \widehat F_1\cong \widehat F_2$. Then 
there exists a tensor functor 
$F_0: \C_1\to \C_2$ such that $F_0\circ F_1\cong F_2$, 
and $F\cong \widehat F_0$. In other words, if tensor functors $G$ and $F\circ G$ are integral and $G$ is surjective then $F$ is integral. 
\end{theorem} 

\begin{proof} We may replace $\C_2$ with ${\rm Im}F_2$, which
is separable by Theorem \ref{quo} or Theorem \ref{subca}, and assume that $F_2,\widehat F_2,F$ are surjective.  
Then by \cite{ENO}, Proposition 5.3, we have an inclusion $(\widehat F_1)^*_{\widehat\C_2}: (\widehat{\C_1})^*_{\widehat\C_2}\hookrightarrow \widehat\D^*_{\widehat\C_2}=\widehat{\D^*_{\C_2}}$. The image of $(\widehat F_1)^*_{\widehat\C_2}$ is then a lift of some multifusion subcategory 
$\B\subset \D^*_{\C_2}$. By Theorem \ref{subca}, $\B$ is separable.  Let $H: \B\hookrightarrow \D^*_{\C_2}$ be the corresponding inclusion functor. 
Then $(\widehat F_1)^*_{\widehat \C_2}$ defines an equivalence \linebreak
$L: (\widehat{\C_1})^*_{\widehat\C_2}\cong \widehat \B$ such that $(\widehat F_1)^*_{\widehat \C_2}=\widehat H\circ L$.  
Then 
$$
\widehat H\circ L\circ F^*_{\widehat\C_2}=(\widehat F_1)^*_{\widehat \C_2}\circ F^*_{\widehat\C_2}=(\widehat F_2)^*_{\widehat \C_2}=\widehat{(F_2)^*_{\C_2}}.
$$ 
Thus, by Theorem \ref{t7a}, $L\circ F^*_{\widehat\C_2}=\widehat{M}$ for some $M: \C_2\to \B$. Dualizing, we get 
$\widehat {M^*_{\C_2}}=F\circ L^*_{\widehat \C_2}$, where $L^*_{\widehat \C_2}: \widehat{\B_{\C_2}^*}\to \widehat{\C}_1$ is an equivalence. 
By Theorem \ref{t5}, $L^*_{\widehat \C_2}=\widehat N$ for some equivalence $N: \B_{\C_2}^*\to \C_1$. Thus, 
$F=\widehat F_0$, where $F_0:=M^*_{\C_2}\circ N^{-1}$. 
\end{proof}

\begin{remark} In spite of Theorems \ref{t7a} and \ref{t7b}, in general 
we don't know if any tensor functor $F: \widehat \C_1\to \widehat \C_2$ between liftings of separable (multi)fusion categories 
is always isomorphic to a lifting of a tensor functor $F_0: \C_1\to \C_2$, even in the case when $\C_2=\Vec_k$ and $\C_1=\Rep H$, where $H$ is a semisimple cosemisimple Hopf algebra over $k$. In this special case, this is the question whether any Drinfeld twist $J$ for $\widehat{H}$ is gauge equivalent to a lifting of a twist $J_0$ for $H$.  

Likewise, we don't know if any fusion category or semisimple cosemisimple Hopf algebra in characteristic zero whose 
(global) dimension is coprime to $p$ descends to characteristic $p$. 
\end{remark}


\begin{thebibliography}{99999} 

\bibitem [DN]{DN}  A.~Davydov, D.~Nikshych, 
{\em The Picard crossed module of a braided tensor category},
Algebra and Number Theory, \textbf{7} (2013), no.\ 6, 1365--1403.

\bibitem[D]{D} P. Deligne, Cat\'egories Tannakiennes, 
The Grothendieck Festschrift, Birkh\"auser, 1990, pp.111---195. 

\bibitem[DSS]{DSS}
C. Douglas, C. Schommer-Pries, N. Snyder,
Dualizable tensor categories,
arXiv:1312.7188. 

\bibitem[EGA]{EGA} EGA IV, El\'ements de g\'eom\'etrie alg\'ebrique. IV.
\'Etude locale des sch\'emas et des
morphismes de sch\'emas, Inst. Hautes \'Etudes Sci. Publ. Math. (1964---67), nos. 20,
24, 28, 32.

\bibitem[EG]{EG} P. Etingof, S. Gelaki,
On finite dimensional semisimple and cosemisimple Hopf 
algebras in positive characteristic, math.QA 9805106, IMRN, v.16(1998),
p.851---864.

\bibitem[EG2]{EG2} P. Etingof, S. Gelaki, 
The classification of triangular semisimple and cosemisimple Hopf algebras over an algebraically closed field,
Int Math Res Notices, 2000 (5): p. 223-234.

\bibitem[EG3]{EG3} P. Etingof, S. Gelaki, Exact sequences of tensor categories with respect to a module category, 
arXiv:1504.01300.

\bibitem[EGNO]{EGNO} P. Etingof, S. Gelaki, D. Nikshych, V. Ostrik, Tensor categories, AMS, Providence, 2015. 

\bibitem[ENO]{ENO} P. Etingof, D. Nikshych, V. Ostrik,  On fusion categories,
math.QA/0203060, Ann. of Math. (2)  162  (2005),  no. 2, 581---642.

\bibitem[ENO2]{ENO2}   P.~Etingof, D.~Nikshych, and V.~Ostrik.
{\em  Fusion categories and homotopy theory}, 
Quantum Topology,   \textbf{1} (2010), no.\ 3,  209---273.

\bibitem[EO]{EO}  P.  Etingof,  V.  Ostrik,
Finite  tensor  categories,  Moscow  Math.  Journal
4(2004), 627---654.

\bibitem[FK]{FK}  V. Franjou, W. van der Kallen,
Power reductivity over an arbitrary base, 
Documenta Mathematica (2010), 
Volume: Extra Vol., page 171---195, arXiv:0806.0787. 

\bibitem[G]{G} F. Grosshans, Algebraic Homogeneous Spaces and Invariant Theory, Lecture Notes in Mathematics, 1673. Springer-Verlag, Berlin, 1997.

\bibitem[Ha]{Ha} W. Haboush, Reductive groups are geometrically reductive, Annals of Mathematics, Vol. 102, No. 1, 102 (1): 67---83, 1975. 

\bibitem[LR]{LR} R. Larson and D. Radford. Finite dimensional cosemisimple Hopf algebras in charac-
teristic 0 are semisimple.
J. Algebra 117
(1988), no. 2, 267---289.

\bibitem[NZ]{NZ} W. Nichols and M. B. Zoeller, A Hopf algebra freeness theorem.
Amer. J. Math. 111 (1989), 381---385.

\bibitem[Se]{Se} C. S. Seshadri, Geometric reductivity over an arbitrary base, 
Advances in Mathematics, v.26, p.225---274, 1977. 

\bibitem[SGA3]{SGA3} M. Demazure, A. Grothendieck, eds., 
S\'eminaire de G\'eom\'etrie Alg\'ebrique du Bois Marie - 1962-64 - 
Sch\'emas en groupes - (SGA 3) - vol. 1,2,3 (Lecture notes in mathematics 151,152,153). Berlin; New York: Springer-Verlag, 1970.

\bibitem[Sp]{Sp} T. Springer, Invariant Theory, Lecture Notes in Mathematics, Vol. 585, Springer-Verlag, Berlin-New York, 1977.

\bibitem[St]{St} D. Stefan, The set of types of $n$-dimensional semisimple and cosemisimple Hopf Algebras
is finite, J. Algebra, 193,
(1997), 571---580.

\bibitem[T]{T} R. Thomason, Equivariant  resolution,  linearization,  and
HilbertÕs fourteenth problem over arbitrary base schemes,
Adv. in Math. 65, no. 1 (1987), 16---34.

\bibitem[vdK]{vdK} W. van der Kallen,
A reductive group with finitely generated cohomology algebras,
in: Algebraic Groups and Homogeneous Spaces, Mumbai 2004, 
Edited by Vikram B. Mehta. Tata Inst. Fund. Res. Stud. Math.
(2007), 301Ð314. Tata Inst. Fund. Res., Mumbai,
arXiv:math/0403361.

\bibitem[YJ]{YJ} D. Yau, M. Johnson, A foundation for PROPs, algebras, and modules, AMS, 2015. 

\end{thebibliography}
\end{document}